\documentclass[12pt]{amsart}
\usepackage{amsthm}
\usepackage{amsmath}
\usepackage{version}
\usepackage{enumerate}
\usepackage{graphicx}
\usepackage[usenames,dvipsnames] {color}

\numberwithin{equation}{section}

\def\d{\mathbb{D}}
\def\c{\mathbb{C}}
\def\t{\mathbb{T}}

\def\calr{\mathcal{R}}

\def\calf{\mathcal{F}}
\newcommand\calm{\mathcal{M}}
\newcommand\calb{\mathcal{B}}

\def\car#1{|#1|_{\rm car}}

\def\be{\begin{equation}}
\def\ee{\end{equation}}

\def\s0{s_0}
\def\p0{p_0}

\newcommand\nd{nondegenerate }

\DeclareMathOperator{\Car}{{\mathrm Car}}

 \newtheorem{theorem}{Theorem}[section]
 \newtheorem{corollary}[theorem]{Corollary}
 \newtheorem{lemma}[theorem]{Lemma}
 \newtheorem{proposition}[theorem]{Proposition}

\newtheorem{definition}[theorem]{Definition}
\newtheorem{remark}[theorem]{Remark}

\newtheorem{fact*}{Fact}

\DeclareMathOperator\rank{rank}

\newcommand\half{{\tfrac 12}}

\newcommand\dd{\mathrm d}

\newcommand\idd{\mathrm{id}_\mathbb{D}}
\newcommand\e{\mathrm e}
\newcommand{\eps}{\varepsilon}

\newcommand\ess{\mathcal{S}}

\newcommand{\ip}[2]{\left\langle #1, #2 \right\rangle}

\newcommand{\inv}{^{-1}}

\newcommand{\ph}{\varphi}
\renewcommand\phi{\varphi}
\newcommand\al{\alpha}

\newcommand\Ga{\Gamma}
\newcommand\de{\delta}

\newcommand\la{\lambda}

\newcommand\beq{\begin{equation}}
\newcommand\ds{\displaystyle}
\newcommand\eeq{\end{equation}}
\newcommand\df{\stackrel{\rm def}{=}}
\newcommand\ii{\mathrm i}

\newcommand\bbm{\begin{bmatrix}}
\newcommand\ebm{\end{bmatrix}}
\newcommand\bpm{\begin{pmatrix}}
\newcommand\epm{\end{pmatrix}}
\numberwithin{equation}{section}

\theoremstyle{definition}

\begin{document}
\title[Carath\'eodory extremal functions] {Carath\'eodory extremal functions on the symmetrized bidisc}
\author{Jim Agler}
\address{Department of Mathematics, University of California at San Diego, CA \textup{92103}, USA}
\author{Zinaida A. Lykova}
\address{School of Mathematics,  Statistics and Physics, Newcastle University, Newcastle upon Tyne
 NE\textup{1} \textup{7}RU, U.K.}
\email{Zinaida.Lykova@ncl.ac.uk}

\author{N. J. Young}
\address{School of Mathematics, Statistics and Physics, Newcastle University, Newcastle upon Tyne NE1 7RU, U.K.
{\em and} School of Mathematics, Leeds University,  Leeds LS2 9JT, U.K.}
\email{Nicholas.Young@ncl.ac.uk}
\date{5th May 2018}

\dedicatory{To Rien Kaashoek in esteem and friendship}



\thanks{Partially supported by National Science Foundation Grants DMS 1361720 and 1665260, the UK Engineering and Physical Sciences Research Council grant  EP/N03242X/1, the London Mathematical Society Grant 41730  and  Newcastle University}

\begin{abstract} 
We show how realization theory can be used to find the solutions of the Carath\'eodory extremal problem on the
 symmetrized bidisc
\[
G \stackrel{\rm{def}}{=} \{(z+w,zw):|z|<1, \, |w|<1\}.
\]
We show that, generically, solutions are unique up to composition with automorphisms of the disc.  We also
obtain formulae for large classes of extremal functions for the Carath\'eodory problems for tangents of non-generic types.
\end{abstract} 

\maketitle
\sloppy
\fussy

\pagenumbering{arabic}
\setcounter{page}{1}

\section*{Introduction}\label{intro}
A constant thread in the research of Marinus Kaashoek over several decades has been the power of realization theory applied to a wide variety of problems in analysis.   Among his many contributions in this area we mention his monograph \cite{bgk}, written with his longstanding collaborators Israel Gohberg and Harm Bart, which was an early and influential work in the area, and his more recent papers \cite{KaavS2014,FHK2014}.  Realization theory uses explicit formulae for functions in terms of  operators on Hilbert space to prove function-theoretic results.

In this paper we continue along the Bart-Gohberg-Kaashoek path by using realization theory to prove results in complex geometry.  Specifically, we are interested in the geometry of the {\em symmetrized bidisc}
\[
G \stackrel{\rm{def}}{=} \{(z+w,zw):|z|<1, \, |w|<1\},
\]
a domain in $\c^2$ that has been much studied in the last two decades: see  \cite{cos04,ez05,jp04,bhatta,sarkar,tryb,aly2016}, along with many other papers.
We shall use realization theory to prove detailed results about the {\em Carath\'eodory extremal problem} on $G$, defined as follows (see \cite{kob98, jp}).

Consider a domain (that is, a connected open set) $\Omega$ in $\c^n$.
For domains $\Omega_1,\ \Omega_2$, we denote by $\Omega_2(\Omega_1)$ the set of holomorphic maps from $\Omega_1$ to $\Omega_2$.
A point in the complex tangent bundle $T\Omega$ of $\Omega$ will be called a {\em tangent} (to $\Omega$).  Thus if $\de\df (\la,v)$ is a tangent to $\Omega$ then $\la\in \Omega$ and $v$ is a point in the complex tangent space $T_\la \Omega \sim \c^n$ of $\Omega$ at $\la$.  We say that $\de$ is a {\em \nd} tangent if $v\neq 0$.  We write $|\cdot|$ for the Poincar\'e metric on $T\d$:
\[
|(z, v)| \df \frac{|v|}{1-|z|^2} \quad \mbox{ for } z\in\d, \, v\in\c.
\]

The {\em Carath\'eodory}  or {\em Carath\'eodory-Reiffen  pseudometric} \cite{jp} on $\Omega$ is the Finsler pseudometric $\car{\cdot}$ on $T\Omega$ defined for $\de=(\la,v)\in T\Omega$ by
\begin{align}\label{carprob}
\car{\de} & \df \sup_{F\in \d(\Omega)} |F_*(\de)| \notag \\
	&= \sup_{F\in \d(\Omega)} \frac{ |D_vF(\la)|}{1-|F(\la)|^2}.
\end{align}
Here $F_*$ is the standard notation for the pushforward of $\de$ by the map $F$ to an element of $T\d$, given by
\[
\ip{g}{F_*(\de)}=\ip{g\circ F}{\de}
\]
for any analytic function $g$ in a neighbourhood of $F(\la)$.

The {\em Carath\'eodory extremal problem}  $\Car \de$ on $\Omega$ is to calculate $\car{\de}$ for a given $\de\in T\Omega$, and to find the corresponding extremal functions, which is to say, the functions $F\in\d(\Omega)$ for which the supremum in equation \eqref{carprob} is attained.   We shall also say that $F$ {\em solves} $\Car\de$ to mean that $F$ is an extremal function for $\Car\de$.

For a general domain $\Omega$ one cannot expect to find either $\car{\cdot}$ or the corresponding extremal functions explicitly.  In a few cases, however, there are more or less explicit formulae for $\car{\de}$.  In particular, when $\Omega=G$, $\car{\cdot}$ is a metric on $TG$ (it is positive for \nd tangents) and  the following result obtains \cite[Theorem 1.1 and Corollary 4.3]{ay2004}.  We use the co-ordinates $(s^1,s^2)$ for a point of $G$.
\begin{theorem}\label{extthm10}
Let $\de$ be a \nd tangent vector in $TG$.  There exists $\omega\in\t$ such that the function in $\d(G)$ given by
\beq\label{defPhi}
\Phi_\omega(s^1,s^2) \df \frac{2\omega s^2-s^1}{2-\omega s^1}
\eeq
is extremal for the Carath\'eodory problem $\Car{\de}$ in $G$.
\end{theorem}
It follows that $\car{\de}$ can be obtained as the maximum modulus of a fractional quadratic function over the unit circle \cite[Corollary 4.4]{ay2004}\footnote{Unfortunately there is an $\omega$ missing in equation (4.7) of \cite{ay2004}.  The derivation given there shows that the correct formula is the present one.}: if $\de=\left((s^1,s^2), v\right) \in TG$ then
\begin{align*}
\car{\de}&= \sup_{\omega\in\t} |(\Phi_\omega)_*(\de)| \\
	&=\sup_{\omega\in\t}\left|\frac{v_1(1-\omega^2 s^2)-v_2\omega(2-\omega s^1)}{(s^1-\overline {s^1} s^2)\omega^2-2(1-|s^2|^2)\omega +\bar s-\overline {s^2} s^1}\right|.
\end{align*}
Hence $\car{\de}$ can easily be calculated numerically to any desired accuracy.
In the latter equation we use superscripts (in $s^1, s^2$) and squares (of $\omega$, $|s^2|$).

The question arises: what are the extremal functions for the problem $\Car{\de}$?
By Theorem \ref{extthm10}, there is an extremal function for $\Car{\de}$ of the form $\Phi_\omega$ for some $\omega$ in $\t$, but are there others?   It is clear that if $F$ is an extremal function for $\Car\de$ then so is $m\circ F$ for any automorphism $m$ of $\d$, by the invariance of the Poincar\'e metric on $\d$.  We shall say that the solution of $\Car\de$ is {\em essentially unique} if, for every pair of extremal functions $F_1,F_2$ for $\Car\de$, there exists an automorphism $m$ of $\d$ such that $F_2=m\circ F_1$.

We show in Theorem \ref{ess1} that, for any \nd tangent $\de\in TG$, if there is a {\em unique} $\omega$ in $\t$ such that $\Phi_\omega$ solves $\Car\de$, then the solution of $\Car\de$ {\em is} essentially unique.  Indeed, for any point $\la\in G$, the solution of $\Car(\la,v)$ is essentially unique for generic directions $v$ (Corollary \ref{genuniq}).  We also derive (in Section \ref{royal}) a parametrization of all solutions of $\Car\de$ in the special case that $\de$ is tangent to the `royal variety' $(s^1)^2=4s^2$ in $G$, and in Sections \ref{flat} and  \ref{purelybalanced} we obtain large classes of Carath\'eodory extremals for two other classes of tangents, called {\em flat} and {\em purely balanced} tangents.

The question of the essential uniqueness of solutions of $\Car\de$ in domains including $G$ was studied by L. Kosi\'nski and W. Zwonek in \cite{kos}.  Their terminology and methods differ from ours; we explain the relation of their Theorem 5.3 to our Theorem \ref{ess1} in Section \ref{relation}.  Incidentally, the authors comment that very little is known about the set of all Carath\'eodory extremals for a given tangent in a domain.  As far as the domain $G$ goes, in this paper we derive a substantial amount of information, even though we do not achieve a complete description of all Carath\'eodory extremals on $G$.

The main tool we use is a model formula for analytic functions from $G$ to the closed unit disc $\d^-$ proved in \cite{AY2017} and stated below as Definition \ref{defGmodel} and Theorem \ref{modelGthm}.  Model formulae and realization formulae for a class of functions are essentially equivalent: one can pass back and forth between them by standard methods (algebraic manipulation in one direction, lurking isometry arguments in the other).

\section{Five types of tangent}\label{5types}

There are certainly \nd tangents $\de\in TG$ for which the solution of $\Car\de$ is not essentially unique. Consider, for example, $\de$ of the form
\[
\de=\left( (2z,z^2), 2c(1,z)\right)
\]
for some $z\in\d$ and nonzero complex $c$.  We call such a tangent {\em royal}: it is tangent to the `royal variety'
\[
\calr\df \{(2z,z^2):z\in\d\}
\]
in $G$.  By a simple calculation, for any $\omega\in\t$,
\[
\Phi_\omega(2z,z^2)= -z, \qquad D_v\Phi_\omega(2z,z^2)=- c,
\]
where $v=2c(1,z)$,
so that $\Phi_\omega(2z,z^2)$ and $ D_v\Phi_\omega(2z,z^2)$ are independent of $\omega$.  It follows from Theorem \ref {extthm10} that $\Phi_\omega$ solves $\Car\de$ for {\em all} $\omega\in\t$ and that
\beq\label{carroyal}\
\car{\de}=\frac{|D_v\Phi_\omega(2z,z^2)|}{1-|\Phi_\omega(2z,z^2)|^2}= \frac{|c|}{1-|z|^2}.
\eeq
Now if $\omega_1,\ \omega_2$ are distinct points of $\t$, there is no automorphism $m$ of $\d$ such that $\Phi_{\omega_1}=m\circ\Phi_{\omega_2}$; this is a consequence of the fact that $(2\bar\omega,\bar\omega^2)$ is the unique singularity of $\Phi_\omega$ in the closure $\Ga$ of $G$.  Hence the solution of $\Car\de$ is not essentially unique.

Similar conclusions hold for another interesting class of tangents, which we call {\em flat}.  These are the tangents of the form
\[
(\la,v)=\left( (\beta+\bar\beta z,z), c(\bar\beta,1)\right)
\]
for some $\beta \in\d$ and $c\in\c\setminus \{0\}$.  It is an entertaining calculation to show that
\beq\label{entertain}
\car{(\la,v)}=\frac{|D_v\Phi_\omega(\la)|}{1-|\Phi_\omega(\la)|^2}=\frac{|c|}{1-|z|^2}
\eeq
for all $\omega\in\t$.
Again, the solution to $\Car(\la,v)$ is far from being essentially unique.

There are also tangents $\de\in TG$ such that $\Phi_\omega$ solves $\Car\de$ for exactly two values of $\omega$ in $\t$; we call these {\em purely balanced} tangents.  They can be described concretely as follows.  For any hyperbolic automorphism $m$ of $\d$ (that is, one that has two fixed points $\omega_1$ and $\omega_2$ in $\t$) let $h_m$ in $G(\d)$ be given by
\[
h_m(z)=(z+m(z),zm(z))
\]
for $z\in\d$.   A purely balanced tangent has the form
\beq\label{pbexpress}
\de=(h_m(z),c h_m'(z))
\eeq
for some hyperbolic automorphism $m$ of $\d$, some $z\in\d$ and some $c\in \c\setminus\{0\}$.  It is easy to see that, for $\omega\in\t$, the composition $\Phi_\omega\circ h_m$ is a rational inner function of degree at most $2$ and that the degree reduces to $1$ precisely when $\omega$ is either $\bar\omega_1$ or $\bar\omega_2$.  Thus, for these two values of $\omega$ (and only these), $\Phi_\omega\circ h_m$ is an automorphism of $\d$. It follows that $\Phi_\omega$ solves $\Car \de$ if and only if $\omega=\bar\omega_1$ or $\bar\omega_2$.

A fourth type of tangent, which we call {\em exceptional}, is similar to the purely balanced type, but differs in that the hyperbolic automorphism $m$ of $\d$ is replaced by a {\em parabolic} automorphism, that is, an automorphism $m$ of $\d$ which has a single fixed point $\omega_1$ in $\t$, which has multiplicity $2$.  The same argument as in the previous paragraph shows that $\Phi_\omega$ solves the Carath\'eodory problem if and only if $\omega=\bar\omega_1$.

The fifth and final type of tangent is called {\em purely unbalanced}.  It consists of the tangents $\de=(\la,v)\in TG$ such that $\Phi_\omega$ solves $\Car \de$ for a unique value $\e^{\ii t_0}$ of $\omega$ in $\t$ and
\beq\label{disting}
 \left.   \frac{d^2}{dt^2} \frac{|D_v\Phi_{\e^{\ii t}}(\la)|}{1-|\Phi_{\e^{\ii t}}(\la)|^2}\right|_{t=t_0} < 0.
\eeq
The last inequality distinguishes purely unbalanced from exceptional tangents -- the left hand side of equation \eqref{disting} is equal to zero for exceptional tangents.

The five types of tangent are discussed at length in our paper \cite{aly2016}.  We proved  \cite[Theorem 3.6]{aly2016} a `pentachotomy theorem', which states that every \nd tangent in $TG$ is of exactly one of the above five types.  We also give, for a representative tangent of each type, a cartoon showing the unique complex geodesic in $G$ touched by the tangent \cite[Appendix B]{aly2016}.

It follows trivially from Theorem \ref{extthm10} that, for every \nd tangent $\de\in TG$, either
\begin{enumerate}
\item[\rm{(1)}] there exists a unique $\omega\in\t$ such that $\Phi_\omega$ solves $\Car\de$, or
\item[\rm{(2)}] there exist at least two values of $\omega$ in $\t$ such that $\Phi_\omega$ solves $\Car\de$.
\end{enumerate}
The above discussion shows that Case (1) obtains for purely unbalanced and exceptional tangents, while Case (2) holds for royal, flat and purely balanced tangents.
For the purpose of this paper, the message to be drawn is that Case (1) is generic in the following sense.  Consider a point $\la\in G$.  Each tangent $v$ in $T_\la G$ has a `complex direction' $\c v$, which is a one-dimensional subspace of $\c^2$, or in other words, a point of the projective space $\mathrm{CP}^2$.  The directions corresponding to the royal  (if any) and flat tangents at $\la$ are just single points in $\mathrm{CP}^2$, while, from the constructive nature of the expression \eqref{pbexpress} for a purely balanced tangent, it is easy to show that there is a smooth one-real-parameter curve of purely balanced directions (see \cite[Section 1]{aly2017}).  It follows that the set of directions $\c v\in \mathrm{CP}^2$  for which a unique $\Phi_\omega$ solves $\Car\de$ contains a dense open set in $\mathrm{CP}^2$.  To summarise:
\begin{proposition}\label{generic}
For every $\la\in G$ there exists a dense open set $V_\la$ in $\mathrm{CP}^2$ such that whenever $\c v \in V_\la$, there exists a unique $\omega\in\t$ such that $\Phi_\omega$ solves $\Car (\la,v)$.
\end{proposition}

\section{Tangents with a unique extremal $\Phi_\omega$}\label{Uniq}

In Section \ref{5types} we discussed extremal functions of the special form $\Phi_\omega, \ \omega \in\t$, for the Carath\'eodory problem in $G$.  However, there is no reason to expect that the $\Phi_\omega$ will be the only extremal functions.  For example, if $\de=(\la,v)$ is a \nd tangent and $\Phi_{\omega_1}, \dots,\Phi_{\omega_k}$ all solve $\Car \de$, then one can generate a large class of other extremal functions as follows.  Choose an automorphism $m_j$ of $\d$ such that $m_j\circ \Phi_{\omega_j}(\la)=0$ and $D_v(m_j\circ\Phi_{\omega_j})(\la_j)> 0$ for  $ j=1,\dots,k$.
Then each $m_j\circ \Phi_{\omega_j}$ solves $\Car\de$, and so does any convex combination of them.

Nevertheless,  if there  is a {\em unique} $\omega\in\t$ such that $\Phi_\omega$ is extremal for $\Car\de$ then the solution of $\Car\de$ {\em is} essentially unique.

\begin{theorem}\label{ess1}
Let $\de$ be a  \nd tangent in $G$ such that $\Phi_\omega$ solves $\Car\de$ for a unique value of $\omega$ in $\t$. If $\psi$ solves $\Car\de$ then there exists an automorphism $m$ of $\d$ such that $\psi=m\circ \Phi_\omega$.
\end{theorem}

 For the proof 
recall  the following model formula \cite[Definition 2.1 and Theorem 2.2]{AY2017}.

\begin{definition}\label{defGmodel}
A $G$-\emph{model}
 for a function $\ph$ on $G$ is a triple $(\calm,T,u)$ where $\calm$ is a separable Hilbert space, $T$ is a contraction acting on $\calm$ and $u:G \to \calm$ is an analytic map such that, for all $s,t\in G$,
\beq\label{modelform}
 1-\overline{\ph(t)}\ph(s)= \ip{ (1-t_T^* s_T) u(s)}{u(t)}_\calm
\eeq
where, for $s\in G$,
\[
s_T \df (2s^2T-s^1)(2-s^1T)\inv.
\]

A $G$-model $(\calm,T,u)$ is {\em unitary} if $T$ is a unitary operator on $\calm$.
\end{definition}

For any domain $\Omega$ we define the {\em Schur class } $\ess(\Omega)$ to be the set of holomorphic maps from $\Omega$ to the closed unit disc $\d^-$.

\begin{theorem}\label{modelGthm}
Let $\ph$ be a function on $G$.  The following three statements are equivalent.
\begin{enumerate}
\item [\rm (1)]  $\ph\in\ess(G)$;
\item [\rm (2)] $\ph$ has a $G$-model;
\item [\rm (3)] $\ph$ has a unitary $G$-model $(\calm, T, u)$.
\end{enumerate}
\end{theorem}

From a $G$-model of a function $\ph\in\ess(G)$ one may easily proceed by means of a standard lurking isometry argument to a realization formula
\[
\ph(s)=A+Bs_T(1-Ds_T)\inv C, \quad \mbox{ all } s\in G,
\]
for $\ph$, where $ABCD$ is a contractive or unitary colligation on $\c\oplus\calm$.  However, for the present purpose it is convenient to work directly from the $G$-model.

We also require a long-established fact about $G$ \cite{ay2004}, related to the fact that the Carath\'eodory and Kobayashi metrics on $TG$ coincide.
\begin{lemma}\label{k=c}
If $\de$ is a \nd tangent to $G$ and $\ph$ solves $\Car\de$ then there exists $k$ in $G(\d)$ such that $\ph\circ k=\idd$.  Moreover, if $\psi$ is any  solution of $\Car\de$ then $\psi\circ k$ is an automorphism of $\d$.
\end{lemma}

We shall need some minor  measure-theoretic technicalities.

\begin{lemma}\label{baspos}
Let $Y$ be a set and let 
\[
A:\t\times Y\times Y \to \c
\]
be a map such that
\begin{enumerate}[\rm (1)]
\item $A(\cdot,z,w)$ is continuous on $\t$ for every $z,w \in Y$;
\item $A(\eta,\cdot,\cdot)$ is a positive kernel on $Y$ for every $\eta\in\t$.
\end{enumerate}
Let $\calm$ be a separable Hilbert space, let $T$ be a unitary operator on $\calm$ with spectral resolution
\[
T=\int_\t \eta \, \dd E(\eta)
\]
and let $v:Y\to\calm$ be a mapping.  
Let
\beq\label{defC}
C(z,w)= \int_\t A(\eta,z,w)\, \ip{\dd E(\eta)v(z)}{v(w)}
\eeq
for all $z,w \in Y$.  Then  $C$ is a positive kernel on $Y$.
\end{lemma}
\begin{proof}
Consider any finite subset $\{z_1,\dots,z_N\}$ of $Y$.  We must show that the $N\times N$ matrix
\[
\bbm C(z_i,z_j) \ebm_{i,j=1}^N
\]
is positive.

Since $A(\cdot,z_i,z_j)$ is continuous on $\t$ for each $i$ and $j$, we may approximate the $N\times N$-matrix-valued function $[A(\cdot,z_i,z_j)]$ uniformly on $\t$ by integrable simple functions of the form
\[
[f_{ij}]= \sum_\ell  b^\ell \chi_{\tau_\ell}
\]
for some $N\times N$ matrices $b^\ell$ and Borel sets $\tau_\ell$, where $\chi$ denotes `characteristic function'.  Moreover we may do this in such a way that
each $b^\ell$ is a value $[A(\eta,z_i,z_j)]$ for some $\eta\in\t$, hence is positive.  Then
\beq\label{approxsum}
\bbm \int_\tau f_{ij}(\eta)\, \ip{\dd E(\eta) z_i}{z_j} \ebm_{i,j=1}^N =\sum_\ell b^\ell * \bbm \ip{E(\tau_\ell)v_i}{v_j} \ebm_{i,j=1}^N
\eeq
where $*$ denotes the Schur (or Hadamard) product of matrices.   Since the matrix $\bbm \ip{E(\tau_\ell)v_i}{v_j} \ebm$ is positive and the Schur product of positive matrices is positive, every approximating sum of the form \eqref{approxsum} is positive, and hence the integral in equation \eqref{defC} is a positive matrix.

\end{proof}

\begin{lemma}\label{meas2}
For $i,j=1,2$ let $a_{ij}:\t\to\c$ be continuous and let each $a_{ij}$ have only finitely many zeros in $\t$. 
Let $\nu_{ij}$ be a complex-valued Borel measure on $\t$ such that, for every Borel set $\tau$ in $\t$,
\[
\bbm \nu_{ij}(\tau) \ebm_{i,j=1}^2 \geq 0.
\]
Let $X$ be a Borel subset of $\t$ and suppose that
\[
 \bbm a_{ij}(\eta) \ebm_{i,j=1}^2\;  \mbox{ is positive and of rank } 2  \mbox{ for all } \eta \in X.
\]
Let
\[
C= \bbm c_{ij}\ebm_{i,j=1}^2
\]
where
\[
c_{ij}= \int_X a_{ij}(\eta)\,  \dd \nu_{ij}(\eta)   \qquad \mbox{ for }i,j=1,2.
\]
If  $\rank C \leq 1$
then either $c_{11}=0$ or $c_{22}=0$.

\end{lemma}
\begin{proof}
By hypothesis the set 
\[
Z\df \bigcup_{i,j=1}^2  \{ \eta\in\t: a_{ij}(\eta)=0\}
\]
is finite.

Exactly as in the proof of Lemma \ref{baspos}, for any Borel set $\tau$ in $\t$,
\beq\label{posint}
\bbm \int_\tau a_{ij}\,  \dd\nu_{ij} \ebm_{ij=1}^2 \geq 0.
\eeq

Suppose that $C$ has rank at most $1$ but $c_{11}$ and $c_{22}$ are both nonzero.  Then there exists a nonzero $2\times 1$ matrix $c=[c_1 \,  c_2]^T$ such that $C=cc^*$ for $i,j=1,2$ and $c_1, \ c_2$ are nonzero.

For any Borel set $\tau\subset X$,
\[
\bbm \int_\tau a_{ij}\, \dd\nu_{ij} \ebm \leq \bbm \int_\tau+\int_{X\setminus\tau} a_{ij}\,  \dd\nu_{ij} \ebm = \bbm \int_{X} a_{ij}\,  \dd\nu_{ij} \ebm= C= cc^*.
\]
Consequently there exists a unique $\mu(\tau) \in [0,1]$ such that
\beq\label{gotmu}
\bbm \int_\tau a_{ij}\, \dd\nu_{ij} \ebm = \mu(\tau) C.
\eeq
It is easily seen that $\mu$ is a Borel probability measure on $X$.  Note that if $\eta\in Z$, say $a_{ij}(\eta)=0$, then
on taking $\tau=\{\eta\}$ in equation \eqref{gotmu}, we deduce that
\[
\mu(\{\eta\})c_i\bar c_j =0.
\]
Since $c_1, c_2$ are nonzero, it follows that $\mu(\{\eta\})=0$.  Hence $\mu(Z)=0$.

Equation \eqref{gotmu} states that $\mu$ is absolutely continuous with respect to $\nu_{ij}$ on $X$ and the Radon-Nikodym derivative is given by
\[
 c_i \bar c_j \frac{\dd \mu}{\dd\nu_{ij} }= a_{ij}
\]
for $i,j=1,2$.
Hence, on $X\setminus Z$,
\beq\label{RNd}
\dd\nu_{ij} = \frac{c_i\bar c_j}{ a_{ij}} \dd\mu, \qquad i,j=1,2.
\eeq

Pick a compact subset $K$ of $X\setminus Z$ such that $\mu(K)>0$.  This is possible, since $\mu(X\setminus Z) =1$ and
Borel measures on $\t$ are automatically regular.  By compactness, there exists a point $\eta_0\in K$ such that, for every open neighbourhood $U$ of $\eta_0$,
\[
\mu(U\cap K) >0.
\]

Notice that, for $\eta\in \t\setminus Z$,
\[
\det\bbm \ds \frac{c_i\bar c_j}{ a_{ij}(\eta)} \ebm_{i,j=1}^2 = -\frac {|c_1c_2|^2\det\bbm a_{ij}(\eta)\ebm}{a_{11}(\eta)a_{22}(\eta)|a_{12}(\eta)|^2} < 0.
\]
Thus $[c_i \bar c_j a_{ij}(\eta_0)\inv]$ has a negative eigenvalue.  Therefore there exists a unit vector $x\in\c^2$, an $\eps >0$  and an open neighourhood $U$ of $\eta_0$ in $\t$ such that
\[
\ip{ \bbm c_i\bar c_j a_{ij}(\eta)\inv\ebm x}{x} < -\eps
\]
for all $\eta\in U$.  We then have
\begin{align*}
\ip { \bbm \nu_{ij}(U\cap K) \ebm x}{x} &= \ip{\int_{U\cap K} \bbm c_i\bar c_j  a_{ij}(\eta)\inv \ebm \dd\mu(\eta)x}{x} \\
	&= \int_{U\cap K} \ip{ \bbm c_i\bar c_j a_{ij}(\eta)\inv \ebm x}{x} \dd\mu(\eta) \\
	& < -\eps \mu(U\cap K) \\
	& < 0.
\end{align*}
This contradicts the positivity of the matricial measure $\bbm \nu_{ij}\ebm$.  Hence either $c_1=0$ or $c_2=0$.

\end{proof}

\begin{proof}[Proof of Theorem {\rm \ref{ess1}}]
Let $\de$ be a \nd tangent to $G$ such that $\Phi_\omega$ is the unique function from the collection $\{\Phi_\eta\}_{\eta\in\t}$ that solves $\Car\de$.   Let $\psi$ be a solution of $\Car\de$.   We must find an automorphism $m$ of $\d$ such that $\psi=m\circ\Phi_\omega$.

By Lemma \ref{k=c}, there exists $k$ in $G(\d)$ such that 
\beq\label{getk}
\Phi_\omega\circ k= \idd,
\eeq
and moreover, the function
\beq\label{gotm}
m\df \psi\circ k
\eeq
is an automorphism of $\d$.  Let
\beq\label{defphi}
\ph=m\inv\circ \psi.
\eeq
Then 
\beq\label{propphi}
\ph\circ k= m\inv\circ \psi\circ k = m\inv\circ m=\idd.
\eeq

  By Theorem \ref{modelGthm}, there is a unitary $G$-model $(\calm, T, u)$ for $\ph$.  By the Spectral Theorem for unitary operators, there is a spectral measure $E(.)$ on $\t$ with values in $\calb(\calm)$ such that 
\[
T= \int_\t \eta \; \dd E(\eta).
\]
Thus, for $s\in G$,
\[
s_T =  (2s^2T-s^1)(2-s^1T)\inv =\int_\t \Phi_\eta(s) \; \dd E(\eta).
\]
Therefore, for all $s,t \in G$,
\begin{align}\label{63.1}
1-\overline{\ph(t)}\ph(s) &= \ip{ (1-t_T^* s_T) u(s)}{u(t)}_\calm \notag \\
	&= \int_\t \left( 1-\overline{\Phi_\eta(t)}\Phi_\eta(s)\right) \ip{\dd E(\eta)u(s)}{u(t)}_\calm.
\end{align}

Consider $z,w\in\d$, put $s=k(z), \, t=k(w)$ in equation \eqref{63.1}.  Invoke equation \eqref{propphi} and divide equation \eqref{63.1} through by  $1-\bar w z$ to obtain, for $z,w\in\d$,
\begin{align}
1 &= \int_{\{\omega\}}+\int_{\t\setminus\{\omega\}} \frac{1-\overline{\Phi_\eta \circ k(w)}\Phi_\eta\circ k(z)}{1-\bar w z} \ip{\dd E(\eta)u\circ k(z)}{u\circ k(w)} \notag \\
	&= I_1+I_2 \label{63.2}
\end{align}
where
\begin{align}\label{defI2}
I_1(z,w)	&= \ip{E(\{\omega\})u\circ k(z)}{u\circ k(w)},\notag\\
I_2(z,w)  &= \int_{\t\setminus \{\omega\}}  \frac{1-\overline{\Phi_\eta \circ k(w)}\Phi_\eta\circ k(z)}{1-\bar w z} \ip{\dd E(\eta)u\circ k(z)}{u\circ k(w)}.
\end{align}
The left hand side $1$ of equation \eqref {63.2} is a positive kernel of rank one on $\d$, and $I_1$ is also a positive kernel.   The integrand in $I_2$ is a positive kernel on $\d$ for each $\eta\in\t$, by Pick's theorem, since $\Phi_\eta\circ k$ is in the Schur class.  Hence, by Lemma \ref{baspos}, $I_2$ is  also a positive kernel on $\d$.  Since $I_1+I_2$ has rank $1$, it follows that $I_2$ has rank at most $1$ as a kernel on $\d$.

By hypothesis, $\Phi_\eta$ does {\em not} solve $\Car\de$ for any $\eta\in \t\setminus\{\omega\}$.  
Therefore $\Phi_\eta\circ k$ is a Blaschke product of degree $2$, and consequently, for any choice of distinct points
$z_1,z_2$ in $\d$, the $2\times 2$ matrix
\beq\label{defaij}
\bbm a_{ij}(\eta) \ebm_{i,j=1}^2\df \bbm  \ds \frac{1-\overline{\Phi_\eta \circ k(z_i)}\Phi_\eta\circ k(z_j)}{1-\bar z_i z_j}\ebm_{i,j=1}^2
\eeq
 is a positive matrix of rank $2$ for every $\eta\in \t\setminus\{\omega\}$.  In particular, $a_{11}(\eta)>0$ for all $\eta\in \t\setminus\{\omega\}$.

Moreover, each $a_{ij}$ has only finitely many zeros in $\t$, as may be seen from the fact that
$a_{ij}$ is a ratio of trigonometric polynomials in $\eta$.  To be explicit, if we temporarily write
$k=(k^1,k^2):\d\to G$, then equation \eqref{defaij} expands to $a_{ij}(\eta)=P(\eta)/Q(\eta)$
where
\begin{align*}
P(\eta) &=4\left(1-\overline{k^2(z_i)}k^2(z_j)\right)-2\eta\left(k^1(z_j)-
\overline{k^1(z_i)}k^2(z_j)\right) \\
	& \hspace*{2cm} -2\bar\eta\left(\overline{k^1(z_i)} - \overline{k^2(z_i)}k^1(z_j)\right), \\
Q(\eta)&= (1-\bar z_iz_j)(2-\eta k^1(z_i))^-(2-\eta k^1(z_j)).
\end{align*}
Let 
\[
\nu_{ij} = \ip{E(\cdot) u\circ k(z_i)}{u\circ k(z_j)}.
\]
Clearly $[\nu_{ij}(\tau)] \geq 0$ for every Borel subset $\tau$ of $ \t\setminus\{\omega\}$.  By definition \eqref{defI2},
\[
I_2(z_i,z_j) = \int_{\t\setminus\{\omega\}} a_{ij} \, \dd \nu_{ij}
\]
for $i,j=1,2$.  Moreover, by equation \eqref{63.2},
\[
[I_2(z_i,z_j)]\leq [I_1(z_i,z_j)]+[I_2(z_i,z_j)] =\bbm 1&1\\1&1\ebm.
\]
It follows that 
\beq\label{I2kap}
\bbm  \int_{\t\setminus\{\omega\}} a_{ij} \, \dd \nu_{ij} \ebm = [I_2(z_i,z_j] = \kappa \bbm 1&1\\1&1\ebm
\eeq
for some $\kappa \in [0,1]$.
We may now apply Lemma \ref{meas2} with $X=\t\setminus \{\omega\}$ to deduce that $\kappa=0$ and hence $I_2(z_i,z_j)=0$.
In particular,
\[
0=I_2(z_1,z_1)= \int_{\t\setminus\{\omega\}} a_{11} \, \dd\nu_{11}.
\]
Since $a_{11}>0$ on $\t\setminus\{\omega\}$, it follows that $\nu_{11}(\t\setminus\{\omega\})=0$, which is to say that
\beq\label{2nd}
E(\t\setminus\{\omega\}) u\circ k(z_1) = 0.
\eeq
Since $z_1,z_2$ were chosen arbitrarily in $\t\setminus\{\omega\}$, we have $I_2 \equiv 0$
and therefore, by equation \eqref{63.2},
\beq\label{2I1}
1= I_1= \ip{E(\{\omega\})u\circ k(z)}{u\circ k(w)}
\eeq
for all $z,w\in\d$.  It follows that 
\[
\|E(\{\omega\})u\circ k(z) -E(\{\omega\})u\circ k(w)\|^2=0
\]
for all $z,w$,
and hence that there exists a unit vector $x\in\calm$ such that
\[
E(\{\omega\})u\circ k(z) =x
\]
for all $z\in\d$.

In equation \eqref{63.1}, choose $t=k(w)$ for some $w\in \d$.  Since $\Phi_\omega\circ k=\idd$, we have for all $s\in G$,
\begin{align*}
1-\bar w\ph(s)&= 1-\overline{\ph\circ k(w)} \ph(s) \\
	&=\int_{\{\omega\}}+\int_{\t\setminus\{\omega\}} \left(1-\overline{\Phi_\eta\circ k(w)}\Phi_\eta(s)\right) \ip{\dd E(\eta)u(s)}{u\circ k(w)} \\
	&=(1-\bar w \Phi_\omega(s))\ip{u(s)}{x} + \\
	 &\hspace*{1cm}	 \int_{\t\setminus\{\omega\}} \left(1-\overline{\Phi_\eta\circ k(w)}\Phi_\eta(s)\right) \ip{\dd E(\eta)u(s)}{u\circ k(w)}.
\end{align*}
In view of equation \eqref{2nd}, the scalar spectral measure in the second term on the right hand side is zero on $\t\setminus \{\omega\}$.  Hence the integral is zero, and so, for all $s\in G$ and $w\in\d$,
\begin{align}\label{64.1}
1-\bar w \ph(s)&= (1-\bar w \Phi_\omega(s))\ip{u(s)}{x}.
\end{align}

Put $w=0$ to deduce that 
\[
\ip{u(s)}{x}=1
\]
 for all $s\in G$, then equate coefficients of  $\bar w$ to obtain $\ph=\Phi_\omega$.
Hence, by equation \eqref{gotm},
\[
\psi=m\circ \ph=m\circ \Phi_\omega
\]
as required.
\end{proof}

On combining Theorem \ref{ess1} and Proposition \ref{generic} we obtain the statement in the abstract.
\begin{corollary}\label{genuniq}
Let $\la\in G$.  For a generic direction $\c v$ in $\mathrm{CP}^2$, the solution of the Carath\'eodory problem $\Car (\la, v)$ is essentially unique.
\end{corollary}

It will sometimes be useful in the sequel to distinguish a particular Carath\'eodory extremal function from a class of functions that are equivalent up to composition with automorphisms of $\d$.
Consider any tangent $\de\in TG$ and any solution $\ph$ of $\Car\de$.
The functions $m\circ\ph$, with $m$ an automorphism of $\d$, also solve $\Car\de$, and among them there is exactly one that has the property
\[
m\circ\ph(\la)=0 \quad \mbox{ and } \quad D_v(m\circ \ph)(\la) > 0,
\]
or equivalently,
\beq\label{special}
(m\circ\ph)_*(\de) = (0,\car{\de}).
\eeq
We shall say that $\ph$ is {\em well aligned at} $\de$ if $\ph_*(\de)=(0,\car{\de})$.

With this terminology the following is a re-statement of Theorem \ref{ess1}.
\begin{corollary}
If $\de$ is a  \nd tangent in $G$ such that $\Phi_\omega$ solves $\Car\de$ for a unique value of $\omega$ in $\t$ then there is a unique well-aligned solution of $\Car\de$.  It is expressible as  $m\circ \Phi_\omega$ for some automorphism $m$ of $\d$.
\
\end{corollary}

\section{Royal tangents}\label{royal}
At the opposite extreme from the tangents studied in the last section are the royal tangents to $G$.  Recall that these have the form
\beq\label{roytgt}
\de=\left((2z,z^2),2c(1,z)\right)
\eeq
for some $z\in\d$ and nonzero complex number $c$.  As we observed in Section \ref{5types},
\[
\car{\de}= \frac{|c|}{1-|z|^2}
\]
and {\em all} $\Phi_\omega, \omega\in\t$, solve $\Car\de$.  In this section we shall describe {\em all} extremal functions for $\Car\de$ for royal tangents $\de$,
not just those of the form $\Phi_\omega$.
\begin{theorem}\label{royalthm}
Let $\de\in TG$ be the royal tangent
\beq\label{roytang}
\de=\left((2z,z^2),2c(1,z)\right)
\eeq
for some $z\in\d$ and $c\in\c\setminus\{0\}$.  A function $\ph\in\d(G)$ solves $\Car\de$ if and only if there exists an automorphism $m$ of $\d$ and $\Psi\in\ess(G)$ such that, for all $s\in G$,
\beq\label{theformula}
\ph(s)=m\left(  \half s^1 +\tfrac 14 ((s^1)^2-4s^2)\frac{\Psi(s)}{1-\half s^1\Psi(s)}\right).
\eeq
\end{theorem}

\begin{proof}
  We shall lift the problem
$\Car\de$ to a Carath\'eodory problem on the bidisc $\d^2$, where we can use the results of \cite{AgM2} on the Nevanlinna-Pick problem on the bidisc.  

Let $\pi:\d^2\to G$ be the `symmetrization map', 
\[
\pi(\la^1,\la^2)=(\la^1+\la^2,\la^1\la^2)
\]
and let $k:\d\to\d^2$ be given by $k(\zeta)=(\zeta,\zeta)$ for $\zeta\in\d$.

Consider the royal tangent $\de$ of equation \eqref{roytang} and let
\[
\de_{zc}=\left((z,z),(c,c)\right) \in T\d^2.
\]
Observe that 
\[
\pi'(\la)= \bbm 1 & 1\\ \la^2 & \la^1 \ebm,
\]
and so
\beq\label{dedezc}
\pi_*(\de_{zc})= \left(\pi(z,z),\pi'(z,z)(c,c)\right)= \left((2z,z^2),2c(1,z)\right)=\de,
\eeq
while
\[
k_*((z,c))=(k(z),k'(z)c)=\left((z,z),(c,c)\right)=\de_{zc}.
\]

Consider any $\ph\in\d(G)$.
Figure 1 illustrates the situation.
\begin{center}
\begin{figure}
\includegraphics {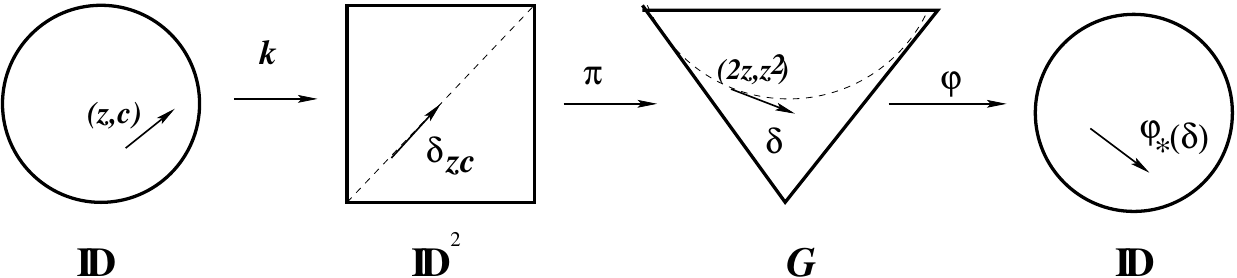}
\caption{}
\end{figure}
\end{center}

It is known that every Carath\'eodory problem on the bidisc is solved by one of the two
co-ordinate functions $F_j(\la)=\la^j$ for $j=1$ or $2$ (for a proof see, for example, \cite[Theorem 2.3]{aly2016}).  Thus
\begin{align*}
\car{\de_{zc}}^{\d^2}&= \max_{j=1,2} \frac{|D_{(c,c)}F_j(z,z)|}{1-|F_j(z,z)|^2} \\
	&=\frac{|c|}{1-|z|^2} \\
	&=\car{\de}.
\end{align*}
Here of course the superscript $\d^2$ indicates the Carath\'eodory extremal problem on the bidisc.

Hence, for $\ph\in\d(G)$,
\begin{align}\label{equiv1}
 \ph\circ\pi\mbox{ solves } \Car\de_{zc} &\iff |(\ph\circ\pi)_*(\de_{zc})|=\frac{|c|}{1-|z|^2} \notag\\
	&\iff |\ph_*\circ \pi_*(\de_{zc})|=\frac{|c|}{1-|z|^2} \quad \mbox{ by the chain rule}\notag\\
	&\iff |\ph_*(\de)|=\frac{|c|}{1-|z|^2} \quad \hspace*{1.2cm}\mbox{ by equation \eqref{dedezc}}  \notag  \\
	&\iff \ph\mbox{ solves }\Car\de.
\end{align}

Next observe that a function $\psi\in\d(\d^2)$ solves $\Car\de_{zc}$ if and only if $\psi\circ k$ is an automorphism of $\d$.  For if $\psi\circ k$ is an automorphism of $\d$ then it satisfies
\[
|(z,c)| = |(\psi\circ k)_*(z,c)|= |\psi_*\circ k_*(z,c)|= |\psi_*(\de_{zc})|,
\]
which is to say that $\psi$ solves $\Car\de_{zc}$.  Conversely, if $\psi$ solves $\Car\de_{zc}$ then $\psi\circ k$ is an analytic self-map of $\d$ that preserves the Poincar\'e metric of a \nd tangent to $\d$, and is therefore (by the Schwarz-Pick lemma) an automorphism of $\d$.  On combining this observation with equivalence \eqref{equiv1} we deduce that 
\begin{align}\label{equiv2}
\ph \mbox{ solves } \Car\de\iff  & \mbox{ there exists an automorphism }m \mbox{ of } \d \notag \\
	& \mbox{ such that } m\inv\circ \ph\circ\pi \circ k= \idd.
\end{align}

For a function $f\in\d(\d^2)$, it is easy to see that $f\circ k=\idd$ if and only if $f$ solves the Nevanlinna-Pick problem
\beq\label{NPD2}
(0,0) \mapsto 0, \qquad (\half,\half) \mapsto \half.
\eeq
See \cite[Subsection 11.5]{AgM2} for the Nevanlinna-Pick problem in the bidisc.  
Hence
\begin{align}\label{equiv3}
\ph \mbox{ solves } \Car\de  & \iff  \mbox{ there exists an automorphism }m \mbox{ of } \d \mbox{ such that } \notag\\
	& m\inv\circ \ph\circ\pi\mbox{ solves the Nevanlinna-Pick problem \eqref{NPD2}.}
\end{align}

In \cite[Subsection 11.6]{AgM2} Agler and McCarthy use realization theory to show the following.

{\em 
A function $f\in\ess(\d^2)$ satisfies the interpolation conditions
\beq\label{interp}
f(0,0)=0,  \qquad f(\half,\half)=\half
\eeq
if and only if there exist $t\in[0,1]$ and  $\Theta$ in the Schur class of the bidisc such that, for all $\la\in\d^2$,
\beq\label{formpsi}
f(\la)= t\la^1+(1-t)\la^2+ t(1-t)(\la^1-\la^2)^2 \frac{\Theta(\la)}{1-[(1-t)\la^1+t\la^2]\Theta(\la)}.
\eeq
}

Inspection of the formula \eqref{formpsi} reveals that $f$ is symmetric
if and only if  $t=\half$ and $\Theta$ is symmetric.  Hence the symmetric functions in $\ess(\d^2)$ that satisfy the conditions
\eqref{interp} are those given by
\begin{align}\label{formf}
f(\la) &=\half \la^1+\half\la^2 + \tfrac 14  (\la^1-\la^2)^2 \frac{\Theta(\la)}{1-\half(\la^1+\la^2)\Theta(\la)}
\end{align}
for some symmetric $\Theta \in\ess(\d^2)$.  Such a $\Theta$ induces a unique function $\Psi\in\ess(G)$ such that
$\Theta=\Psi\circ\pi$, and we may write the symmetric solutions $f$ of the problem \eqref{interp} in the form
$f=\tilde f\circ\pi$ where, for all $s=(s^1,s^2)$ in $G$,
\beq\label{chipsi}
\tilde f(s)= \half s^1 +\tfrac 14 ((s^1)^2-4s^2)\frac{\Psi(s)}{1-\half s^1\Psi(s)}.
\eeq

Let $\ph$ solve $\Car\de$.  By the equivalence \eqref{equiv3}, there exists an automorphism $m$ of $\d$
such that $ m\inv\circ \ph\circ\pi$ solves the Nevanlinna-Pick problem \eqref{NPD2}.  
Clearly $m\inv\circ\ph\circ\pi$ is symmetric.
Hence there exists  $\Psi\in\ess(G)$ such that, for all $\la\in\d^2$,
\beq\label{formpsibis}
m\inv\circ\ph(s)=  \half s^1 +\tfrac 14 ((s^1)^2-4s^2)\frac{\Psi(s)}{1-\half s^1\Psi(s)}.
\eeq
Thus $\ph$ is indeed given by the formula \eqref{theformula}.

Conversely, suppose that for some automorphism $m$ of $\d$ and $\Psi\in\ess(G)$, a function $\ph$ is defined by
equation \eqref{theformula}.  Let $f=m\inv\circ\ph\circ\pi$  Then $f$ is given by the formula \eqref{formf},
where $\Theta=\Psi\circ\pi$. Hence $f$ is a symmetric function that satisfies the interpolation conditions \eqref{interp}.
By the equivalence \eqref{equiv3}, $\ph$ solves $\Car\de$.
\end{proof}

\section{Flat tangents}\label{flat}
In this section we shall give a description of a large class of Carath\'eodory extremals for a flat tangent.
Recall that a flat tangent has the form
\beq\label{aflatgeo}
\de=\left( (\beta+\bar\beta z,z), c(\bar\beta,1)\right)
\eeq
for some $z\in\d$ and $c\neq 0$, where $\beta\in\d$.  Such a tangent touches the `flat geodesic'
\[
F_\beta\df \{ (\beta+\bar\beta w, w):w\in\d\}.
\]
The description depends on a remarkable property of sets of the form $\calr\cup F_\beta, \ \beta\in\d$: they have the norm-preserving extension property in $G$ \cite[Theorem 10.1]{aly2016}.  That is, if $g$ is any bounded analytic function on the variety $\calr\cup F_\beta$, then there exists an analytic function $\tilde g$ on $G$ such that $g=\tilde g|\calr\cup F_\beta$ and the supremum norms of $g$ and $\tilde g$ coincide.  Indeed, the proof of \cite[Theorem 10.1]{aly2016} gives an explicit formula for one such $\tilde g$ in terms of a Herglotz-type integral.  Let us call the norm-preserving extension $\tilde g$ of $g$ constructed in \cite[Chapter 10]{aly2016} the {\em special extension} of $g$ to $G$.

It is a simple calculation to show that $\calr$ and $F_\beta$ have a single point in common.

By equation \eqref{entertain}, for $\de$ in equation \eqref{aflatgeo}
\[
\car{\de} =\frac{|c|}{1-|z|^2}.
\]

\begin{theorem}
Let $\de$ be the flat tangent
\beq\label{aflatgeo2}
\de=\left( (\beta+\bar\beta z,z), c(\bar\beta,1)\right)
\eeq
to $G$, where $\beta\in\d$ and $c\in\c\setminus\{0\}$.  Let $\zeta,\eta$ be the points in $\d$ such that
\[
(2\zeta,\zeta^2)=(\beta+\bar\beta\eta, \eta) \in  \calr\cap F_\beta
\]
and let $m$ be the unique automorphism of $\d$ such that  
\[
m_*((z,c))=(0,\car{\de}).
\]  

For every function $h\in\ess(\d)$ such that $h(\zeta)=m(\eta)$ the special extension $\tilde g$ to $G$ of the function
\beq\label{defg}
g: \calr\cup F_\beta\to \d, \quad (2w,w^2)\mapsto h(w), \quad (\beta+\bar\beta w,w) \mapsto m(w)
\eeq
for $w\in\d$ is a well-aligned Carath\'eodory extremal function for $\de$.
\end{theorem}
\begin{proof}
First observe that  there is indeed a unique automorphism $m$ of $\d$ such that $m_*((z,c))=(0,\car{\de})$, by the Schwarz-Pick Lemma.  
Let
 \[
k(w)=(\beta+\bar\beta w,w) \quad \mbox{ for } w\in\d,
\]
so that $F_\beta=k(\d)$ and $k_*((z,c))=\de$.   By the definition \eqref{defg} of $g$, $g\circ k=m$.

Consider any function $h\in\ess(\d)$ such that $h(\zeta)=m(\eta)$.  By \cite[Lemma 10.5]{aly2016}, the function $g$ defined by equations \eqref{defg} is analytic on $\calr\cup F_\beta$. 

We claim that the special extension $\tilde g$ of  $ g$ to $G$ is a well-aligned Carath\'eodory extremal function for $\de$.  By \cite[Theorem 10.1]{aly2016}, $\tilde g \in\d(G)$.  Moreover
\begin{align*}
(\tilde g)_*(\de) &=(\tilde g)_*\circ k_*((z,c)) \\
	&=(\tilde g\circ k)_*((z,c)) \\
	&=(g\circ k)_*((z,c)) \\
	&=m_*((z,c))\\
	&= (0,\car{\de})
\end{align*}
as required. Thus the Poincar\'e metric of $(\tilde g)_*(\de)$ on $T\d$ is
\[
|(\tilde g)_*(\de)| = |(0,\car{\de})| = \car{\de}. 
\]
Therefore 
$(\tilde g)_*$ is a well aligned Carath\'eodory extremal function for $\de$.
\end{proof}
Clearly the map $g\mapsto \tilde g$ is injective, and so this procedure yields a large class of Carath\'eodory extremals for $\de$, parametrized by the Schur class.

\begin{remark} \rm
In the converse direction, if $\ph$ is any well-aligned Carath\'eodory extremal for $\de$, then $\ph$ is a norm-preserving extension of its restriction to $\calr\cup F_\beta$, which is a function of the type \eqref{defg}.  Thus the class of all well-aligned Carath\'eodory extremal functions for $\de$ is given by the set of
 norm-preserving analytic extensions to $G$ of $g$ in equation \eqref{defg}, as $h$ ranges over functions in the Schur class taking the value $m(\eta)$ at $\zeta$.  Typically there will be many such extensions of $g$, as can be seen from the proof of  \cite[Theorem 10.1]{aly2016}.  An extension is obtained as the Cayley transform of a function defined by a Herglotz-type integral with respect to a probability measure $\mu$ on $\t^2$. In the proof of  \cite[Lemma 10.8]{aly2016}, $\mu$ is chosen to be the product of two measures $\mu_\calr$ and $\mu_\calf$ on $\t$; examination of the proof shows that one can equally well choose any measure $\mu$ on $\t^2$ such that
\[
\mu(A\times\t) = \mu_\calr(A), \quad \mu(\t\times A)= \mu_\calf(A)\quad \mbox{ for all Borel sets } A \mbox{ in }\t.
\]
Thus each choice of $h\in\ess(\d)$ satisfying $h(\zeta)=m(\eta)$ can be expected to give rise to many well-aligned Carath\'eodory extremals for $\de$.
\end{remark}

\section{Purely balanced tangents}\label{purelybalanced}
In this section we find a large class of Carath\'eodory extremals for purely balanced tangents in $G$ by exploiting an embedding of $G$ into the bidisc.

\begin{lemma}\label{injective}
Let
\[
\Phi=(\Phi_{\omega_1}, \Phi_{\omega_2}): G \to \d^2
\]
where $\omega_1,\omega_2$ are distinct points in $\t$.  Then
$\Phi$ is an injective map from $G$ to $\d^2$.
\end{lemma}
\begin{proof}
Suppose $\Phi$ is not injective.  Then there exist distinct points $(s^1,s^2)$, $ (t^1,t^2)\in G$ such that $\Phi_{\omega_j}(s^1,s^2)=\Phi_{\omega_j}(t^1,t^2)$ for $j = 1, 2$.  On expanding and simplifying this relation we deduce that
\[
s^1-t^1-2\omega_j(s^2-t^2) -\omega_j^2(s^1 t^2-t^1 s^2)=0.
\]
A little manipulation demonstrates that both $(s^1,s^2)$ and $(t^1,t^2)$ lie on the complex line
\[
\ell\df\{(s^1,s^2)\in\c^2: (\omega_1 +\omega_2)s^1 -2\omega_1\omega_2 s^2 =2\}.
\]
However, $\ell$ does not meet $G$.  For suppose that $(s^1,s^2)\in \ell\cap G$.
Then there exists $\beta\in\d$ such that
\begin{align*}
s^1&=\beta+\bar\beta s^2,\\
2\omega_1\omega_2 s^2&= (\omega_1 +\omega_2)s^1-2 = (\omega_1 +\omega_2)(\beta+\bar\beta s^2)-2.
\end{align*}
On solving the last equation for $s^2$ we find that
\[
s^2=-\bar\omega_1\bar\omega_2\frac{2-(\omega_1+\omega_2)\beta}{2-(\bar\omega_1+\bar\omega_2)\bar\beta},
\]
whence $|s^2|=1$, contrary to the hypothesis that $(s^1,s^2)\in G$.  Hence $\Phi$ is injective on $G$.
\end{proof}
\begin{remark} \rm
$\Phi$ has an analytic extension to the set $\Ga\setminus\{(2\bar\omega_1,\bar\omega_1^2),(2\bar\omega_2,\bar\omega_2^2)\}$, where $\Ga$ is the closure of $G$ in $\c^2$.  However this extension is {\em not} injective: it takes the constant value $(-\bar\omega_2,-\bar\omega_1)$ on a curve lying in $\partial G$.
\end{remark}

\begin{theorem}\label{purebalextremals}
Let $\de=(\la,v)$ be a purely balanced tangent to $G$ and let $\Phi_\omega$ solve $\Car\de$ for the two distinct points $\omega_1, \omega_2 \in\t$.    Let $m_j$ be the automorphism of $\d$ such that $m_j\circ \Phi_{\omega_j}$ is well aligned at $\de$ for $j=1,2$ and let
\beq\label{defPhi2}
\Phi=(\Phi^1,\Phi^2)=(m_1\circ\Phi_{\omega_1},m_2\circ \Phi_{\omega_2}):G\to\d^2.
\eeq  

For every $t\in [0,1]$ and every function $\Theta$ in the Schur class of the bidisc the function
\begin{align} \label{formextrem}
F&=t\Phi^1+(1-t)\Phi^2+\notag \\
	&\hspace*{1cm}t(1-t)(\Phi^1-\Phi^2)^2\frac{\Theta\circ\Phi}{1-[(1-t)\Phi^1+t\Phi^2]\Theta\circ\Phi}
\end{align}
is a well-aligned Carath\'eodory extremal function for $\de$.
\end{theorem}
\begin{proof}
By Lemma \ref{injective}, $\Phi$ maps $G$ injectively into $\d^2$.  By choice of $m_j$,
\[
(m_j\circ\Phi_{\omega_j})_*(\de)=(0, \car{\de}).
\]
Hence 
\[
\Phi_*(\de)= \left((0,0),\car{\de}(1,1)\right),
\]
which is tangent to the diagonal $\{(w,w):w\in\d\}$ of the bidisc.  Since the diagonal is a complex geodesic in $\d^2$, we have
\[
\car{\Phi_*(\de)} = (0,\car{\de}).
\]
As in Section \ref{royal}, we appeal to \cite[Subsection 11.6]{AgM2} to assert that, for every $t\in[0,1]$ and every function $\Theta$ in the Schur class of the bidisc, the function $f\in\c(\d^2)$ given by
\beq\label{formulaf}
f(\la)= t\la^1+(1-t)\la^2+t(1-t)^2\frac{ \Theta(\la)}{1-[(1-t)\la^1+t\la^2]\Theta(\la)}
\eeq
solves $\Car(\Phi_*(\de))$.  For every such $f$ the function 
$ F\df f\circ\Phi: G\to \d$ satisfies
\[
F_*(\de) = (f\circ\Phi)_*(\de)=f_*(\Phi_*(\de))=(0,\car{\de}).
\]
Thus $F$ is a well-aligned Carath\'eodory extremal for $\de$.
On writing out $F$ using equation \eqref{formulaf} we obtain equation \eqref{formextrem}.
\end{proof}
\begin{remark} \rm
The range of $\Phi$ is a subset of $\d^2$ containing $(0,0)$ and is necessarily nonconvex, by virtue of a result of Costara \cite{cos04} to the effect that $G$ is not isomorphic to any convex domain.    $\Phi(G)$ is open in $\d^2$, since the Jacobian determinant of $(\Phi_{\omega_1},\Phi_{\omega_2})$  at $(s^1,s^2)$ is
\[
	\frac{4(\omega_1-\omega_2)(1-\omega_1\omega_2 s^2)}{(2-\omega_1s^1)^2(2-\omega_2s^1)^2}
\]
which has no zero in $G$.
 Carath\'eodory extremals $F$ given by equation \eqref{formulaf}  have the property that the map $F\circ \Phi\inv$ on $\Phi(G)$ extends analytically to a map in $\d(\d^2)$.  There may be other Carath\'eodory extremals $\ph$ for $\de$ for which $\ph\circ\Phi\inv$ does not so extend.  Accordingly we do not claim that the Carath\'eodory extremals described in Theorem \ref{purebalextremals} constitute all extremals for a purely balanced tangent.
\end{remark}

\section{Relation to a result of L. Kosi\'nski and W. Zwonek}\label{relation}
Our main result in Section \ref{Uniq}, on the essential uniqueness of solutions of $\Car\de$ for purely unbalanced and exceptional tangents, can be deduced from  \cite[Theorem 5.3]{kos} and some known facts about the geometry of $G$.  However, the terminology and methods of Kosi\'nski and Zwonek are quite different from ours, and we feel it is worth explaining their statement in our terminology.

Kosi\'nski and Zwonek speak of left inverses of complex geodesics where we speak of Carath\'eodory extremal functions for \nd tangents.  These are essentially equivalent notions.  By a {\em complex geodesic} in $G$ they mean a holomorphic map from $\d$ to $G$ which has a holomorphic left inverse.  Two complex geodesics $h$ and $k$ are {\em equivalent} if there is an automorphism $m$ of $\d$ such that $h=k\circ m$, or, what is the same, if $h(\d)=k(\d)$.  It is known (for example \cite[Theorem A.10]{ay2004}) that, for every \nd tangent $\de$ to $G$, there is a {\em unique} complex geodesic $k$ of $G$ up to equivalence such that $\de$ is tangent to $k(\d)$.  A function $\ph\in\d(G)$ solves $\Car\de$ if and only if $\ph\circ k$ is an automorphism of $\d$.  Hence, for any complex geodesic $k$ and any \nd tangent $\de$ to $k(\d)$,  to say that $k$ has a unique left inverse up to equivalence is the same as to say that $\Car\de$ has an essentially unique solution.

Kosi\'nski and Zwonek also use a different classification of types of complex geodesics (or equivalently tangent vectors) in $G$, taken from \cite{pz05}.  There it is shown that every complex geodesic $k$ in $G$, up to composition with automorphisms of $\d$ on the right and of $G$ on the left, is of one of the following types.
\begin{enumerate}[\rm (1)]
\item 
\[
k(z)= (B(\sqrt{z})+B(-\sqrt{z}), B(\sqrt{z}) B(-\sqrt{z})
\]
where $B$ is a non-constant Blaschke product of degree $1$ or $2$ satisfying $B(0)=0$;
\item
\[
k(z)= (z+m(z),zm(z))
\]
where $m$ is an automorphism of $\d$ having no fixed point in $\d$.
\end{enumerate}
These types correspond to our terminology from \cite{aly2016} (or from Section \ref{5types}) in the following way.
Recall that an automorphism of $\d$ is either the identity, elliptic, parabolic or hyperbolic, meaning that the set $\{z\in\d^-:m(z)=z\}$ consists of either all of $\d^-$, a single point of $\d$, a single point of $\t$ or two points in $\t$.
\begin{enumerate}
\item[\rm (1a)] If $B$ has degree $1$, so that $B(z)=cz$ for some $c\in\t$ then, up to equivalence, $k(z)=(0,-c^2z)$.  These we call the {\em flat geodesics}.
The general tangents to flat geodesics are the flat tangents described in Section \ref{5types}, that is $\de=\left((\beta+\bar\beta z,z),c(\bar\beta,1)\right)$ for some $\beta\in\d, \, z\in\d$ and nonzero $c\in\c$.
\item[\rm (1b)]  If $B(z)=cz^2$ for some $c\in\t$ then $k(z)= (2cz, c^2z^2)$.  Thus $k(\d)$ is the royal variety $\calr$, and the tangents to $k(\d)$ are the royal tangents.

\item[\rm (1c)] If $B$ has degree $2$ but  is not of the form (1b),  say $B(z)=cz(z-\al)/(1-\bar\al z)$ where $c\in\t$ and $\al\in\d\setminus \{0\}$, then
\[
k(z)= \frac{\left(2c(1-|\al|^2)z, c^2z(z-\al^2)\right)}{1-\bar\al^2z}.
\]
Here $k(\d)$ is not $\calr$ but it meets $\calr$ (at the point $(0,0)$).  It follows that $k(\d)$ is a purely unbalanced geodesic and the tangents to $k(\d)$ are the purely unbalanced tangents.
\item[\rm (2a)]  If $m$ is a hyperbolic automorphism of $\d$ then $k(\d)$ is a purely balanced geodesic and its tangents are purely balanced tangents.
\item[\rm (2b)]  If $m$ is a parabolic automorphism of $\d$ then $k(\d)$ is an exceptional geodesic, and its tangents are exceptional tangents. 
\end{enumerate}
With this description, Theorem 5.3 of \cite{kos} can be paraphrased as stating that a complex geodesic $k$ of $G$ has a unique left inverse (up to equivalence) if and only if $k$ is of one of the forms (1c) or (2b).  These are precisely the purely unbalanced and exceptional cases in our terminology, that is, the cases of tangents $\de$ for which there is a unique $\omega\in\t$ such that $\Phi_\omega$ solves $\Car\de$, in agreement with our Theorem \ref{ess1}.

The authors prove their theorem with the aid of a result of Agler and McCarthy on the uniqueness of solutions of $3$-point Nevanlinna-Pick problems on the bidisc \cite[Theorem 12.13]{AgM2}.  They also use the same example from Subsection 11.6 of \cite{AgM2} which we use for different purposes in Sections \ref{royal} and \ref{purelybalanced}.

\bibliography{references}

\end{document}